\documentclass[12pt,twoside]{amsart}
\usepackage{amsmath}
\usepackage{amsfonts}
\usepackage{amssymb}
\usepackage{color}
\newcommand{\NN}{\mathbb N} 
\newcommand{\RR}{\mathbb R}

\newtheorem{theorem}{Theorem}
\newtheorem{corollary}[theorem]{Corollary}
\newtheorem{lemma}[theorem]{Lemma}

\begin{document}
\renewcommand{\baselinestretch}{1.05}
\title{A phase-field approximation\\ of the Willmore flow\\ with volume constraint} 

\author{Pierluigi Colli}
\address{Dipartimento di Matematica ``F.~Casorati'', Universit\`a di Pavia, Via Ferrata 1, I-27100~Pavia,~Italy}
\email{pierluigi.colli@unipv.it}
\author{Philippe Lauren\c cot}
\address{Institut de Math\'ematiques de Toulouse, CNRS UMR~5219, Universit\'e de Toulouse, F--31062 Toulouse Cedex 9, France} 
\email{laurenco@math.univ-toulouse.fr}
\date{\today}

\begin{abstract}
The well-posedness of a phase-field approximation to the Willmore flow with volume constraint is established. The existence proof relies on the underlying gradient flow structure of the problem: the time discrete approximation is solved by a variational minimization principle.
\end{abstract}

\maketitle

%
%
\pagestyle{myheadings}
\newcommand\testopari{\sc{Pierluigi Colli and Philippe Lauren\c cot}}
\newcommand\testodispari{\sc{A phase-field approximation of the Willmore flow with volume constraint}}
\markboth{\testopari}{\testodispari}

\section{Introduction}\label{sec:int}


Let $\Omega$ be an open bounded subset of $\RR^N$, $1\le N\le 3$, with smooth boundary $\Gamma$. We are interested in the following evolution problem
\begin{eqnarray}
\partial_t v & - & \Delta \mu + (j+\sigma)''(v)\ \mu - \overline{(j+\sigma)''(v)\ \mu} = 0 \,, \quad (t,x)\in (0,\infty)\times\Omega\,, \label{a1} \\
\mu & = & - \Delta v + (j+\sigma)'(v)\,, \quad (t,x)\in (0,\infty)\times\Omega\,, \label{a2} \\
\nabla v\cdot \nu & = & \nabla \mu \cdot \nu = 0\,, \quad (t,x)\in (0,\infty)\times\Gamma\,, \label{a3} \\ 
v(0) & = & v_0\,, \quad x\in \Omega\,, \label{a4}
\end{eqnarray} 
where the nonlinearity $j+\sigma$ is a smooth double well-potential (for instance, $(j+\sigma)(r)=(r^2-1)^2/4$), $\nu$ is the outward unit normal vector field to $\Gamma$, and $\overline{f}$ denotes the spatial mean value of an integrable function $f$, namely,
$$
\overline{f} := \frac{1}{|\Omega|}\ \int_\Omega f(x)\ dx \;\;\mbox{ for }\;\;  f\in L^1(\Omega)\,.
$$
As one can easily realize from \eqref{a1} and \eqref{a3} by integrating over $\Omega$, the mean value of $v$ is conserved during the evolution, that is, $\overline{v}(t) = \overline{v_0}$. 

 The initial-boundary value problem \eqref{a1}-\eqref{a4} is a phase-field approximation of the Willmore flow (cf., in particular, \cite{DLW04, DLRW05}) which belongs to a class of geometric evolutions of hypersurfaces involving nonlinear functions of the principal curvatures of the hypersurface. Recall that the Willmore flow \emph{with volume constraint} for a family of (smooth) hypersurfaces $(\Sigma(t))_{t\ge 0}$ reads 
\begin{equation}\label{spip}
\mathcal{V} = - \Delta_\Sigma H - \frac{H}{2}\ (H^2-4K) + \lambda\,,
\end{equation}
where $\mathcal{V}$, $H$, $K$, and $\Delta_\Sigma$ denote the normal velocity of $\Sigma$, the sum of its principal curvatures (scalar mean curvature), the product of its principal curvatures (Gau\ss{} curvature), and the Laplace-Beltrami operator on $\Sigma$, respectively, while $\lambda$ is the Lagrange multiplier accounting for the volume conservation
$$
\int_\Sigma \mathcal{V}\ ds = 0\,.
$$
In addition, the Willmore flow is the $L^2$-gradient flow of the Willmore energy 
\begin{equation}\label{spirou}
\mathcal{E}_W(\Sigma) := \int_\Sigma H^2\ ds\,.
\end{equation}
Related geometric evolution flows involve more complicated energies such as the Helfrich energy and additional constraints, for instance on the area, and are found in the modelling of biological cell membranes. We refer, e.g., to \cite{BGN08,BMxx,DG91,DLW04,DLRW05,RS06} and the references therein for a more detailed description of these flows and their applications. To our knowledge, the energetic phase-field approximation \eqref{a1}-\eqref{a4} has been introduced in \cite{DLW04} in order to describe the deformation of a vesicle membrane under the elastic bending energy, with prescribed bulk volume and surface area, a related model without constraints being considered in \cite{LM00}.  Here, we restrict our analysis to the case of only the volume constraint, leaving the more complex case of two constraints as in \cite{DLW04} to a subsequent investigation. A nice feature of \eqref{a1}-\eqref{a4} already reported in \cite{DLW04} is that it inherits the gradient flow structure of the Willmore flow and it is actually a gradient flow in $L^2(\Omega)$ for the functional 
\begin{equation}\label{fantasio}
E(v) := \frac{1}{2} \int_\Omega \left[ - \Delta v(x) + (j+\sigma)'(v(x)) \right]^2\ dx\,,
\end{equation}
a property which is a cornerstone of the forthcoming analysis. The connection between the minimizers of the Willmore energy \eqref{spirou} and those of a suitably rescaled version of the energy \eqref{fantasio} of the stationary phase-field model has been investigated in \cite{DG91,Mo05,RS06}, and we refer to \cite{DLW04,DLRW05,Wa08} for the analysis of the relationship between the phase-field approach \eqref{a1}-\eqref{a4} and the Willmore flow, with or without volume and surface constraints. However, the well-posedness of the phase-field approximation does not seem to have been considered so far, and the aim of this note is to show the well-posedness of \eqref{a1}-\eqref{a4} under suitable assumptions on the data: more precisely, we assume that there is $C_0>0$ such that
\begin{eqnarray}
& & j\in\mathcal{C}^3(\RR) \;\mbox{ is a convex function with }\; j(0)=j'(0)=0\,, \label{a6} \\
& & \sigma\in\mathcal{C}^3(\RR) \;\mbox{ with }\; \sigma''\in L^\infty(\RR)\,, \label{a7} \\
& &  j+\sigma\ge 0 \;\mbox{ and }\; r\ (j+\sigma)'(r) \ge -C_0\,, \quad r\in\RR\,. \label{a8}
\end{eqnarray}

Next, owing to the already mentioned expected time invariance of the spatial mean-value of solutions to \eqref{a1}-\eqref{a4}, for $\alpha\in\RR$ we define the functional space
\begin{equation}
W := \left\{ w\in H^2(\Omega)\ :\ \nabla w\cdot \nu = 0 \;\;\mbox{ on }\;\; \Gamma \right\}\quad \mbox{ and its subset }  \quad W_\alpha := \left\{ w\in W\ :\ \overline{w} = \alpha \right\}\,. \label{a5}
\end{equation}

The paper is devoted to the proof of the following existence and uniqueness  result.

\begin{theorem}\label{th:a1}
Given $\alpha\in\RR$ and $v_0\in W_\alpha$, there is a unique solution $v$ to \eqref{a1}-\eqref{a4} satisfying
$$
v\in \mathcal{C}([0,T];L^2(\Omega))\cap L^\infty(0,T;W_\alpha) \;\;\mbox{ and }\;\; \mu  := -\Delta v + (j+\sigma)'(v) \in L^2(0,T;W)
$$
for all $T>0$. In addition, 
\begin{eqnarray}
& & t \longmapsto E\left( v(t) \right) := \frac{1}{2} || \mu(t)||_2^2 \;\;\mbox{ is a non-increasing function }\,, \label{a9} \\
& & \int_0^\infty \left\| -\Delta\mu(t) + (j+\sigma)''(v(t))\ \mu(t) - \overline{(j+\sigma)''(v)\ \mu}(t) \right\|_2^2\ dt \le 2 E(v_0)\,. \label{a10}
\end{eqnarray}
\end{theorem}

Owing to  the above mentioned  gradient flow structure, a classical approach to existence is to use an implicit time scheme and solve a minimization problem at each step, see, e.g., \cite{AGS08} or \cite[Chap.~8]{Vi03}.  The existence of a minimizer to the corresponding stationary problem is discussed in Section~\ref{sec:ex}, and Subsection~\ref{sec:tef} also collects some properties of the auxiliary variable $\mu$. The time discretization is next implemented in Subsection~\ref{sec:td} and convergence of the time discrete scheme is proved in Subsection~\ref{sec:cv}  with the help of monotonicity and compactness properties. Finally, uniqueness is shown in Section~\ref{sec:un} by a standard contraction argument.

\section{Existence}\label{sec:ex}

\subsection{The energy functional}\label{sec:tef}

Following \cite{DLW04}, we define the functional $E$ on $W$ by
\begin{equation}
E(w) := \frac{1}{2} \int_\Omega \left[ - \Delta w(x) + (j+\sigma)'(w(x)) \right]^2\ dx\,. \label{b0}
\end{equation}	
Observe that $E$ is well defined for any $w\in W$ thanks to the continuous embedding of $H^2(\Omega)$ in $L^\infty(\Omega)$, \eqref{a6}, and \eqref{a7}. Indeed, for $w\in W$, we have $w\in L^\infty(\Omega)$ and 
$$
\left| (j+\sigma)'(w) \right| \le \int_0^w j''(r)\ dr + |\sigma'(0)| + \|\sigma''\|_\infty |w| \le |\sigma'(0)| + \left( \sup_{[-\|w\|_\infty,\|w\|_\infty]}{\left\{ j'' \right\}} + \|\sigma''\|_\infty \right)\ |w|\,.
$$
Consequently, $(j+\sigma)'(w)\in L^2(\Omega)$ and $E$ is well defined. We gather some properties of $E$ in the next lemma.

\begin{lemma}\label{le:b0}
Given $\alpha\in\RR$, there is $C_1(\alpha)>0$ depending only on $\Omega$, $\sigma$, $C_0$ in \eqref{a8}, and $\alpha$ such that
\begin{equation}\label{b1}
\|w\|_{H^2} + \|j'(w)\|_2 \le C_1(\alpha)\ \left( 1 + \sqrt{E(w)} \right)\, \quad \mbox{for all} \quad w\in W_\alpha\,.
\end{equation}
\end{lemma}

\begin{proof} Consider $w\in W_\alpha$ and put $\mu:= -\Delta w + (j+\sigma)'(w)$. Then $\mu\in L^2(\Omega)$ with $\|\mu\|_2^2=2E(w)$, and we infer from \eqref{a8} that 
$$
\int_\Omega w\ \mu\ dx = \|\nabla w\|_2^2 + \int_\Omega w\ (j+\sigma)'(w)\ dx \ge \|\nabla w\|_2^2 -C_0\ |\Omega|\,.
$$
Combining the above inequality with the Poincar\'e-Wirtinger inequality
\begin{equation}\label{b2}
\| w - \overline{w} \|_2 \le C_2\ \|\nabla w\|_2\,,
\end{equation}
we obtain
\begin{eqnarray*}
\|\nabla w\|_2^2 & \le & C_0 |\Omega| + \int_\Omega w\ \mu\ dx  \le C_0 |\Omega| + \| w\|_2 \|\mu\|_2 \\
& \le & C_0 |\Omega| + \sqrt{2 E(w)}\ \left( \alpha |\Omega|^{1/2} + \|w-\alpha\|_2 \right) \le  C_0 |\Omega| + \sqrt{2 E(w)}\ \left( \alpha |\Omega|^{1/2} + C_2\ \|\nabla w\|_2 \right) \\ 
& \le &  C_0 |\Omega| + \alpha |\Omega|^{1/2}\ \sqrt{2 E(w)} + \frac{1}{2}\ \|\nabla w\|_2^2 + C_2^2\ E(w)\,,
\end{eqnarray*}
hence $\|\nabla w\|_2^2 \le C(\alpha)\ (1+E(w))$. Using again \eqref{b2}, we conclude that 
\begin{equation}\label{b3}
\| w \|_{H^1} \le C(\alpha)\ \left( 1 + \sqrt{E(w)} \right)\,.
\end{equation}
Now, $w\in W$ solves $-\Delta w + j'(w) = \mu - \sigma'(w)$ and, owing to the monotonicity of $j'$, a classical monotonicity argument shows that
$$
\|\Delta w\|_2 + \|j'(w)\|_2 \le \| \mu - \sigma'(w)\|_2\,.
$$ 
It then follows from \eqref{a7} that
$$
\|\Delta w\|_2 + \|j'(w)\|_2 \le \| \mu \|_2 + |\sigma'(0)| |\Omega|^{1/2} + \|\sigma''\|_\infty\ \|w\|_2\,,
$$
which, together with \eqref{b3}  and $\| \mu \|_2 = \sqrt{2E(w)}$, gives \eqref{b1}. 
\end{proof}

\medskip

Next, given $\tau>0$ and $f\in L^2(\Omega)$, we define the functional $F_{\tau,f}$ on $W$ by
\begin{equation}\label{b4}
F_{\tau,f}(w) := \frac{1}{2}\ \| w-f\|_2^2 + \tau\ E(w)\,, \quad w\in W\,.
\end{equation}

\begin{lemma}\label{le:b1}
Given $\alpha\in\RR$, the functional $F_{\tau,f}$ has (at least) a minimizer in $W_\alpha$.
\end{lemma}

\begin{proof} We set $F:=F_{\tau,f}$ to simplify notations. Since $E$ is nonnegative, $F$ is obviously nonnegative and there is a minimizing sequence $(w_n)_{n\ge 1}$ in $W_\alpha$ such that 
\begin{equation}\label{b5}
m_\alpha := \inf_{w\in W_\alpha}{\{ F(w) \}} \le F(w_n) \le m_\alpha + \frac{1}{n}\,, \quad n\ge 1\,.
\end{equation}
Since $F(w_n)\ge \tau\ E(w_n)$, we readily infer from \eqref{b5} that $(E(w_n))_{n\ge 1}$ is bounded, a property which in turn implies that $(w_n)_{n\ge 1}$ is bounded in $H^2(\Omega)$ by Lemma~\ref{le:b0}. Owing to the compactness of the embedding of $H^2(\Omega)$ in $\mathcal{C}(\bar{\Omega})$, we deduce that there are $w\in H^2(\Omega)$ and a subsequence of $(w_n)_{n\ge 1}$ (not relabeled) such that 
\begin{equation}\label{b6}
w_n \longrightarrow w \;\mbox{ in }\; \mathcal{C}(\bar{\Omega}) \;\mbox{ and }\; w_n \rightharpoonup w \;\mbox{ in }\; H^2(\Omega)\,.
\end{equation}
Clearly, the first convergence implies that $\left( (j+\sigma)'(w_n)\right)_{n\ge 1}$ converges towards $(j+\sigma)'(w)$ in $L^2(\Omega)$ and therefore 
$$
F(w) \le \liminf_{n\to\infty} F(w_n) \le m_\alpha\,.
$$
As $w$ obviously belongs to $W_\alpha$ by \eqref{b6}, we also have $F(w)\ge m_\alpha$ and $w$ is a minimizer of $F$ in $W_\alpha$. 
\end{proof}

\medskip

We next derive an energy inequality and the Euler-Lagrange equation satisfied by minimizers of $F_{\tau,f}$ in $W_\alpha$ when $\overline{f}=\alpha$.

\begin{lemma}\label{le:b2}
Consider $\alpha\in\RR$ and a minimizer $w$ of $F_{\tau,f}$ in $W_\alpha$. Assume further that $\overline{f}=\alpha$. Then $\mu:= -\Delta w + (j+\sigma)'(w)$ belongs to $W$, 
\begin{equation}\label{b7}
\int_\Omega \left[ \frac{w-f}{\tau} - \Delta\mu + (j+\sigma)''(w)\ \mu - \overline{(j+\sigma)''(w)\ \mu} \right] \psi\ dx = 0 \quad \mbox{for all} \quad \psi\in W\,,
\end{equation}
and
\begin{equation}\label{b8}
\left\| -\Delta \mu + (j+\sigma)''(w)\ \mu - \overline{(j+\sigma)''(w)\ \mu} \right\|_2 \le \frac{\|w - f\|_2}{\tau}\,.
\end{equation}
\end{lemma}

\begin{proof}
We set 
$$
\mu := -\Delta w + (j+\sigma)'(w)\,.
$$
Consider $\varepsilon\in (0,1)$ and $\varphi\in W_0$. As $w+\varepsilon\varphi$ belongs to $W_\alpha$, we have $F_{\tau,f}(w)\le F_{\tau,f}(w+\varepsilon\varphi)$ from which we deduce by classical arguments (after passing to the limit as $\varepsilon\to 0$) that
$$
\frac{1}{\tau}\ \int_\Omega (w-f)\ \varphi\ dx + \int_\Omega \mu\ \left( -\Delta\varphi + (j+\sigma)''(w)\ \varphi \right)\ dx \ge 0\,.
$$
Since the above inequality is valid for $\varphi$ and $-\varphi$, we actually have the identity 
\begin{equation}\label{b9}
\frac{1}{\tau}\ \int_\Omega (w-f)\ \varphi\ dx + \int_\Omega \mu\ \left( -\Delta\varphi + (j+\sigma)''(w)\ \varphi \right)\ dx = 0
\end{equation}
for all $\varphi\in W_0$. Now, if $\psi\in W$, the function $\psi-\overline{\psi}$ belongs to $W_0$ and it follows from \eqref{b9} that
\begin{equation}\label{b10}
\frac{1}{\tau}\ \int_\Omega (w-f)\ \psi\ dx + \int_\Omega \mu\ \left( -\Delta\psi + (j+\sigma)''(w)\ \psi \right)\ dx = \overline{(j+\sigma)''(w)\ \mu}\ \int_\Omega \psi\ dx\,,
\end{equation}
since $w$ and $f$ have the same mean value $\alpha$.  Since $\mu \in L^2(\Omega)$ solves the variational equality \eqref{b10} for all test functions $\psi\in W$, we deduce that $\mu\in W$ and satisfies \eqref{b7}.

Next, for $\eta\in (0,1)$, let $\varphi_\eta$ be the unique solution in $W_0$ to 
$$
\varphi_\eta - \eta\ \Delta\varphi_\eta = -\Delta \mu + (j+\sigma)''(w)\ \mu - \overline{(j+\sigma)''(w)\ \mu} \;\;\mbox{ in }\;\; \Omega\,,
$$
the right-hand side of the previous equation being in $L^2(\Omega)$ since $\mu\in W$ and $w\in H^2(\Omega)$ is bounded. Also, the right-hand side of the previous equation has a zero mean-value so that $\varphi_\eta\in W_0$. Taking $\psi=\varphi_\eta$ in \eqref{b7}, we realize that
$$
\int_\Omega \left[ \frac{w-f}{\tau} + \varphi_\eta - \eta\ \Delta\varphi_\eta \right] \varphi_\eta\ dx = 0\,,
$$ 
from which we deduce that
$$
\|\varphi_\eta\|_2^2  \le  \|\varphi_\eta\|_2^2 + \eta\ \|\nabla\varphi_\eta\|_2^2 = - \int_\Omega \frac{w-f}{\tau}\ \varphi_\eta\ dx \le \frac{\|w-f\|_2}{\tau}\ \|\varphi_\eta\|_2 \,,
$$
whence
$$
\|\varphi_\eta\|_2 \le \frac{\|w-f\|_2}{\tau}\,.
$$
Since $(\varphi_\eta)_\eta$ converges toward $(-\Delta \mu + (j+\sigma)''(w)\ \mu - \overline{(j+\sigma)''(w)\ \mu} )$ in $L^2(\Omega)$ as $\eta\to 0$, \eqref{b8} follows from the above inequality.
\end{proof}

\subsection{Time discretization}\label{sec:td}

Let $\alpha\in\RR$ and take an initial  condition $v_0\in W_\alpha$. We consider a positive time step $\tau\in (0,1)$ and define a sequence $(v_n^\tau)_{n\ge 1}$ inductively as follows:
\begin{eqnarray}
& & v_0^\tau := v_0\,, \label{c1} \\
& & v_{n+1}^\tau \;\mbox{ is a minimizer of }\; F_{\tau,v_n^\tau} \;\mbox{ in }\; W_\alpha\,, \quad n\ge 0\,, \label{c2}
\end{eqnarray}
the functional $F_{\tau,v_n^\tau}$ being defined in \eqref{b4}. Setting
\begin{equation}\label{c3b}
\mu_n^\tau := - \Delta v_n^\tau + (j+\sigma)'(v_n^\tau) \;\;\mbox{ and }\;\; M_n^\tau:= \overline{(j+\sigma)''(v_n^\tau)\ \mu_n^\tau}\,,
\end{equation}
we define three piecewise constant time-dependent functions $v^\tau$, $\mu^\tau$, and $M^\tau$ by 
\begin{equation}\label{c3}
\left( v^\tau(t) , \mu^\tau(t) , M^\tau(t) \right) := \left( v_n^\tau , \mu_n^\tau , M_n^\tau \right) \;\;\mbox{ for }\;\; t\in [n\tau,(n+1)\tau) \;\;\mbox{ and }\;\; n\ge 0\,.
\end{equation}

\begin{lemma}\label{le:c1} For $\tau\in (0,1)$, $t_1\ge 0$, and $t_2>t_1$, we have
\begin{eqnarray}
& & E\left( v^\tau(t_2) \right) \le E\left( v^\tau(t_1) \right) \le E(v_0)\,, \label{c4} \\
& & \|v^\tau(t_2) -v^\tau(t_1)\|_2^2 \le 2E(v_0)\ (\tau+t_2-t_1)\,, \label{c5} \\
& & \int_\tau^\infty \left\| -\Delta\mu^\tau(t) + (j+\sigma)''(v^\tau(t))\ \mu^\tau(t) - M^\tau(t) \right\|_2^2\ dt \le 2 E(v_0)\,. \label{c6}
\end{eqnarray}
\end{lemma}

\begin{proof}
Consider $n\ge 0$. Since $v_n^\tau\in W_\alpha$, we infer from \eqref{c2} that $F_{\tau,v_n^\tau}(v_{n+1}^\tau)\le F_{\tau,v_n^\tau}(v_n^\tau)$, that is,
\begin{equation}
\frac{1}{2\tau}\ \left\| v_{n+1}^\tau - v_n^\tau \right\|_2^2 + E\left( v_{n+1}^\tau \right) \le E\left( v_n^\tau \right)\,. \label{c7}
\end{equation}
Let $t_2> t_1\ge 0$ and put $n_i:=[t_i/\tau]$ (the integer part of $t_i/\tau$), $i=1,2$. On the one hand, $n_2\ge n_1$ and it readily follows from \eqref{c7} by induction that
$$
E\left( v^\tau(t_2) \right) = E\left( v_{n_2}^\tau \right) \ \le\ E\left( v_{n_1}^\tau \right) = E\left( v^\tau(t_1) \right)\,,
$$
whence \eqref{c4}. In particular, we have
\begin{equation}
\frac{1}{2} \,  \sup_{t\ge 0} \| \mu^\tau(t) \|_2^2 = \sup_{t\ge 0} E\left( v^\tau(t) \right) = \sup_{n\ge 0} E\left( v_n^\tau \right) \le E\left( v_0^\tau \right) = E(v_0)\,.\label{c8}
\end{equation}
On the other hand, summing \eqref{c7} over $n\in\NN$ gives
\begin{equation}
\frac{1}{2\tau}\ \sum_{n=0}^\infty \left\| v_{n+1}^\tau - v_n^\tau \right\|_2^2 \le E\left( v_0^\tau \right) = E(v_0)\,,\label{c9}
\end{equation}
from which we deduce that
\begin{eqnarray*}
\left\|v^\tau(t_2) - v^\tau(t_1) \right\|_2 & = & \left\|v_{n_2}^\tau - v_{n_1}^\tau \right\|_2 \le \sum_{n=n_1}^{n_2-1} \left\|v_{n+1}^\tau - v_n^\tau \right\|_2 \\
& \le & \left( n_2 - n_1 \right)^{1/2}\ \left( \sum_{n=n_1}^{n_2-1} \left\|v_{n+1}^\tau - v_n^\tau \right\|_2^2 \right)^{1/2} \\
& \le & \left( 1 + \frac{t_2-t_1}{\tau} \right)^{1/2}\ \left( 2\tau E(v_0) \right)^{1/2} \\
& \le & \sqrt{2E(v_0)}\ \left( \tau + (t_2-t_1) \right)^{1/2}\,,
\end{eqnarray*}
and thus \eqref{c5}. Finally, for $n\ge 0$, we have  $\overline{v_{n+1}^\tau} = \overline{v_n^\tau}=\alpha$ by \eqref{c2} and we infer from \eqref{b8} that
$$
\left\| -\Delta \mu_{n+1}^\tau + (j+\sigma)''(v_{n+1}^\tau)\ \mu_{n+1}^\tau - M_{n+1}^\tau \right\|_2 \le \frac{\|v_{n+1}^\tau - v_n^\tau\|_2}{\tau}\,.
$$
Combining \eqref{c9} and the previous inequality give
\begin{eqnarray*}
& & \int_\tau^\infty \left\| -\Delta \mu^\tau(t) + (j+\sigma)''(v^\tau(t))\ \mu^\tau(t) - M^\tau(t) \right\|_2^2\ dt \\ 
& \le & \sum_{n=0}^\infty \int_{(n+1)\tau}^{(n+2)\tau} \left\| -\Delta \mu_{n+1}^\tau + (j+\sigma)''(v_{n+1}^\tau)\ \mu_{n+1}^\tau - M_{n+1}^\tau \right\|_2^2\ dt \\
& \le & \sum_{n=0}^\infty \frac{\|v_{n+1}^\tau - v_n^\tau\|_2^2}{\tau} \le 2 E(v_0) \,,
\end{eqnarray*}
and the proof is complete. 
\end{proof}

Useful bounds on $(v^\tau)_\tau$ and $(\mu^\tau)_\tau$ follow from Lemma~\ref{le:c1}.

\begin{corollary}\label{co:c2}
For all $T>0$, there is $C_3(T)>0$ depending only on $\alpha$, $v_0$, $j$, $\sigma$, and $T$ such that , for $\tau\in (0,1)\cap (0,T)$, 
\begin{eqnarray}
\sup_{t\in [0,T]} \left\| v^\tau(t) \right\|_{H^2} & \le & C_3(T)\,, \label{c10} \\
\int_\tau^T \left( \left\| \mu^\tau(t) \right\|_{H^1}^4 + \left\| \mu^\tau(t) \right\|_{H^2}^2 \right)\ dt & \le & C_3(T)\,. \label{c11}
\end{eqnarray}
\end{corollary}

\begin{proof} The boundedness \eqref{c10} of $(v^\tau)_\tau$ is a straightforward consequence of \eqref{b1} and \eqref{c8}. Next, owing to the continuous embedding of $H^2(\Omega)$ in $L^\infty(\Omega)$ and \eqref{c10}, the  family
$((j+\sigma)''(v^\tau))_\tau$ is bounded in 
$L^\infty((0,T)\times\Omega)$ which, together with \eqref{c8}, imply that 
\begin{equation}
((j+\sigma)''(v^\tau) \mu^\tau)_\tau \;\;\mbox{ is bounded in }\;\; L^\infty(0,T;L^2(\Omega))\,.\label{c12}
\end{equation} 
Setting $f^\tau := -\Delta \mu^\tau + (j+\sigma)''(v^\tau)\ \mu^\tau - M^\tau$, it follows from \eqref{c6} and \eqref{c12} that 
\begin{eqnarray*}
\left( \int_\tau^T \|\Delta \mu^\tau(t)\|_2^2\ dt \right)^{1/2} & = & \left( \int_\tau^T \left\| (j+\sigma)''(v^\tau(t))\ \mu^\tau(t) - M^\tau(t) - f^\tau(t) \right\|_2^2\ dt \right)^{1/2} \\
& \le &  2\ \left( \int_\tau^T \left\| (j+\sigma)''(v^\tau(t))\ \mu^\tau(t) \right\|_2^2\ dt \right)^{1/2} + \left( \int_\tau^T \left\| f^\tau(t) \right\|_2^2\ dt \right)^{1/2} \\
& \le & C(T)\,,
\end{eqnarray*}
which gives the boundedness of $(\mu^\tau)_\tau$ in $L^2(\tau,T;H^2(\Omega))$ with the help of \eqref{c8}. Finally, $\mu^\tau\in W$ and solves 
$$
-\Delta\mu^\tau + j''(v^\tau)\ \mu^\tau = f^\tau - \sigma''(v^\tau)\ \mu^\tau +M^\tau \;\;\mbox{ in }\;\; \Omega\,.
$$
Taking the scalar product in $L^2(\Omega)$ of the previous equation with $\mu^\tau$ and using the nonnegativity of $j''$ due to the convexity \eqref{a6} of $j$  and the boundedness \eqref{a7} of $\sigma''$, we obtain
\begin{eqnarray*}
\|\nabla\mu^\tau\|_2^2 & \le & \|\nabla\mu^\tau\|_2^2 + \int_\Omega j''(v^\tau)\ (\mu^\tau)^2\ dx \\
& \le & \|f^\tau\|_2\ \|\mu^\tau\|_2 + \|\sigma''\|_\infty\ \|\mu^\tau\|_2^2 + |M^\tau|\ \|\mu^\tau\|_2 \,.
\end{eqnarray*}
We next deduce from \eqref{c8} and \eqref{c12} that 
$$
\|\nabla\mu^\tau\|_2^2 \le C(T)\ \left( 1 + \|f^\tau\|_2 \right)\,,
$$
and the boundedness of the right-hand side of the above inequality in $L^2(\tau,T)$ follows at once from \eqref{c6}. 
\end{proof}

\subsection{Convergence}\label{sec:cv}

Owing to \eqref{c5}, \eqref{c10}, and the compactness of the embedding of $H^2(\Omega)$ in $\mathcal{C}(\bar{\Omega})$, a refined version of the Ascoli-Arzel\`a theorem (in the spirit of \cite[Prop.~3.3.1]{AGS08}) ensures that $(v^\tau)_\tau$ is relatively compact in $\mathcal{C}([0,T]\times\bar{\Omega})$ for all $T>0$. Consequently, there are three functions $v$, $\mu$, and $M$ and a subsequence $\left( v^{\tau_k} \right)_{k\ge 1}$ of $(v^\tau)_\tau$ such that, for all $T>0$, 
$$
v\in \mathcal{C}([0,T]\times\bar{\Omega})\cap L^\infty(0,T;H^2(\Omega))\,, \quad \mu\in L^\infty(0,T;L^2(\Omega))\,, \quad M\in L^\infty(0,T)\,,
$$
and
\begin{eqnarray}
v^{\tau_k} & \longrightarrow & v \;\;\mbox{ in }\;\; \mathcal{C}([0,T]\times\bar{\Omega})\,, \label{c14} \\
v^{\tau_k} & \stackrel{*}{\rightharpoonup} & v \;\;\mbox{ in }\;\; L^\infty(0,T;H^2(\Omega))\,, \label{c15} \\
\mu^{\tau_k} & \stackrel{*}{\rightharpoonup} & \mu \;\;\mbox{ in }\;\; L^\infty(0,T;L^2(\Omega))\,, \label{c16} \\
M^{\tau_k} & \stackrel{*}{\rightharpoonup} & M \;\;\mbox{ in }\;\; L^\infty(0,T)\,. \label{c17} 
\end{eqnarray}
Thanks to the smoothness of $j$ and $\sigma$ and the convergences \eqref{c14}--\eqref{c17}, it is straightforward to pass to the limit in \eqref{c3b} and conclude that 
\begin{equation}
\mu = - \Delta v + (j+\sigma)'(v) \;\;\mbox{ and }\;\; M = \overline{(j+\sigma)''(v)\ \mu}\,. \label{c18}
\end{equation}
In addition, \eqref{c11}, \eqref{c16}, and a lower semicontinuity argument guarantee that 
\begin{equation}
\mu\in L^4(0,T;H^1(\Omega)) \cap L^2(0,T;H^2(\Omega)) \quad \mbox{for all} \quad T>0\,. \label{c19}
\end{equation}

It remains to derive the equation solved by $v$. Let $\psi\in W$, $t>0$, $n=[t/\tau]$, and $m\in\{0,\ldots,n-1\}$. Using the definition of $v_{m+1}^\tau$ and Lemma~\ref{le:b2}, we are led to 
$$
\int_\Omega \left[ \frac{v_{m+1}^\tau-v_m^\tau}{\tau} - \Delta\mu_{m+1}^\tau + (j+\sigma)''(v_{m+1}^\tau)\ \mu_{m+1}^\tau - M_{m+1}^\tau \right] \psi\ dx = 0\,,
$$
which also reads
$$
\int_\Omega \left( v_{m+1}^\tau-v_m^\tau \right)\ \psi\ dx = \int_{(m+1)\tau}^{(m+2)\tau} \int_\Omega \left[ \Delta\mu^\tau(s) - (j+\sigma)''(v^\tau(s))\ \mu^\tau(s) + M^\tau(s) \right] \psi\ dxds\,.
$$
Summing the above identities over $m\in \{0,\ldots,n-1\}$, we obtain
$$
\int_\Omega \left( v_n^\tau-v_0^\tau \right)\ \psi\ dx = \int_\tau^{(n+1)\tau} \int_\Omega \left[ \Delta\mu^\tau(s) - (j+\sigma)''(v^\tau(s))\ \mu^\tau(s) + M^\tau(s) \right] \psi\ dxds\,.
$$
$$
\int_\Omega \left( v^\tau(t)-v_0 \right)\ \psi\ dx = \int_\tau^{(n+1)\tau} \int_\Omega \left[ \Delta\mu^\tau(s) - (j+\sigma)''(v^\tau(s))\ \mu^\tau(s) + M^\tau(s) \right] \psi\ dxds\,.
$$
Noticing that $t\le (n+1)\tau \le t+\tau$, we may take $\tau=\tau_k$ in the above identity and pass to the limit as $k\to\infty$ with the help of \eqref{c14}--\eqref{c17} to obtain
\begin{equation}
\int_\Omega \left( v(t)-v_0 \right)\ \psi\ dx = \int_0^t \int_\Omega \left[ \Delta\mu(s) - (j+\sigma)''(v(s))\ \mu(s) + M(s) \right] \psi\ dxds\,.\label{c20}
\end{equation}
Collecting \eqref{c18}-\eqref{c20} completes the proof of the existence part of Theorem~\ref{th:a1}. The properties \eqref{a9} and \eqref{a10}  next follow from \eqref{c4}, \eqref{c6}, and the convergences \eqref{c14}-\eqref{c17}. 

\section{Uniqueness}\label{sec:un}

Let $v_1$ and $v_2$ be two solutions to \eqref{a1}-\eqref{a4} with $\mu_i:=-\Delta v_i + (j+\sigma)'(v_i)$ and $M_i := \overline{(j+\sigma)''(v_i) \mu_i}$, $i=1,2$. Fix $T>0$. Since $H^2(\Omega)$ is continuously embedded in $L^\infty(\Omega)$, the regularity properties of $v_1$, $v_2$, $\mu_1$, and $\mu_2$ listed in Theorem~\ref{th:a1} ensures that there is $K>0$ depending on $T$ such that
\begin{equation}
\sup_{t\in [0,T]} \left( \|v_1(t)\|_\infty + \|v_2(t)\|_\infty + \|\mu_1(t)\|_2 + \|\mu_2(t)\|_2 \right) + \int_0^T \left( \|\mu_1(s)\|_\infty^2 + \|\mu_2(s)\|_\infty^2 \right)\ ds \le K\,. \label{d1}
\end{equation}
It then follows from \eqref{d1} and the smoothness of $j$ and $\sigma$ that 
\begin{eqnarray}
& & \hskip-1cm \left| (j+\sigma)''(v_1)\ \mu_1 - (j+\sigma)''(v_2)\ \mu_2 \right|  \label{d2}\\ 
& \le & \left| (j+\sigma)''(v_1) - (j+\sigma)''(v_2) \right|\ \left| \mu_1 \right| +  \left| (j+\sigma)''(v_2) \right|\ \left| \mu_1 - \mu_2 \right| \nonumber \\
& \le & \left\| (j+\sigma)''' \right\|_{L^\infty(-K,K)}\ \left| v_1 - v_2 \right|\ \left| \mu_1 \right| +  \left\| (j+\sigma)'' \right\|_{L^\infty(-K,K)}\ \left| \mu_1 - \mu_2 \right|\,, \nonumber \\
& \le & C\ \left( \left| \mu_1 \right|\ \left| v_1 - v_2 \right| + \left| \mu_1 - \mu_2 \right| \right)\,,\nonumber 
\end{eqnarray}
from which we deduce that 
\begin{eqnarray}
|M_1-M_2| & \le & \frac{1}{|\Omega|}\ \int_\Omega \left| (j+\sigma)''(v_1)\ \mu_1 - (j+\sigma)''(v_2)\ \mu_2 \right|\ dx \label{d3} \\
& \le & C\ \int_\Omega \left( \left| \mu_1 \right|\ \left| v_1 - v_2 \right| + \left| \mu_1 - \mu_2 \right| \right)\ dx \nonumber \\
& \le & C\ \left( \|\mu_1\|_2\ \|v_1-v_2\|_2 + \|\mu_1-\mu_2\|_2 \right)\,. \nonumber
\end{eqnarray}
Since $v_1-v_2$ solves 
$$
\partial_t (v_1-v_2) - \Delta (\mu_1-\mu_2) = M_1-M_2 - (j+\sigma)''(v_1)\ \mu_1 + (j+\sigma)''(v_2)\ \mu_2
$$
and $v_1-v_2$ and $\mu_1-\mu_2$ both belong to $W$, we have
\begin{eqnarray*}
\frac{1}{2}\ \frac{d}{dt} \|v_1-v_2\|_2^2 & = & \int_\Omega (\mu_1-\mu_2)\ \Delta (v_1-v_2)\ dx + \int_\Omega (M_1-M_2)\ (v_1-v_2)\ dx \\
& - & \int_\Omega \left[ (j+\sigma)''(v_1)\ \mu_1 - (j+\sigma)''(v_2)\ \mu_2 \right]\ (v_1-v_2)\ dx \,.
\end{eqnarray*}
We deduce from \eqref{a2}, \eqref{d1}, \eqref{d2}, and \eqref{d3} that 
\begin{eqnarray*}
\frac{1}{2}\ \frac{d}{dt} \|v_1-v_2\|_2^2 & = & \int_\Omega (\mu_1-\mu_2)\ \left[ (j+\sigma)'(v_1)-(j+\sigma)'(v_2) - (\mu_1-\mu_2) \right]\ dx  \\
& + & \int_\Omega (M_1-M_2)\ (v_1-v_2)\ dx \\
& - & \int_\Omega \left[ (j+\sigma)''(v_1)\ \mu_1 - (j+\sigma)''(v_2)\ \mu_2 \right]\ (v_1-v_2)\ dx \\
& \le & \|(j+\sigma)''\|_{L^\infty(-K,K)}\ \|\mu_1-\mu_2\|_2\ \|v_1-v_2\|_2 - \|\mu_1-\mu_2\|_2^2 \\
&  + & C\ \left( \|\mu_1\|_2\ \|v_1-v_2\|_2 + \|\mu_1-\mu_2\|_2 \right)\ \|v_1-v_2\|_2 \\
& + & C\ \int_\Omega \left( \left| \mu_1 \right|\ \left| v_1 - v_2 \right| + \left| \mu_1 - \mu_2 \right| \right)\ |v_1-v_2|\ dx \\
& \le & C\ \|\mu_1-\mu_2\|_2\ \|v_1-v_2\|_2 - \|\mu_1-\mu_2\|_2^2 \\
& + & C\ \left( 1+ \|\mu_1\|_\infty \right)\ \|v_1-v_2\|_2^2 \\
& \le & C\ \left( 1+ \|\mu_1\|_\infty \right)\ \|v_1-v_2\|_2^2\,.
\end{eqnarray*}
Therefore, recalling \eqref{d1}, 
$$
\|(v_1-v_2)(t)\|_2^2 \le \|(v_1-v_2)(0)\|_2^2\ \exp{\left( C\ \int_0^t \left( 1+\|\mu_1(s)\|_\infty \right)\ ds \right)} \le C\ \|(v_1-v_2)(0)\|_2^2
$$
for $t\in [0,T]$, and the uniqueness assertion follows.

\section*{Acknowledgments}

This work was initiated during a visit of the first author at the 
Institut de Math\'ematiques de Toulouse, Universit\'e Paul Sabatier,
whose financial support and kind hospitality are gratefully acknowledged. 


\end{document}